\documentclass[11pt]{amsart}

\usepackage{amsmath}
\usepackage{amsfonts}
\usepackage{amssymb}
\usepackage{latexsym}
\usepackage{amscd}
\usepackage{citesort}
\usepackage{mathrsfs}
\usepackage[all]{xy}

\newdimen\AAdi%
\newbox\AAbo%
\def\AAk#1#2{\s_etbox\AAbo=\hbox{#2}\AAdi=\wd\AAbo\kern#1\AAdi{}}%
\def\AAr#1#2#3{\s_etbox\AAbo=\hbox{#2}\AAdi=\ht\AAbo\raise#1\AAdi\hbox{#3}}%
\font\tenmsb=msbm10 at 12pt \font\sevenmsb=msbm7 at 8pt
\font\fivemsb=msbm5 at 6pt
\newfam\msbfam
\textfont\msbfam=\tenmsb \scriptfont\msbfam=\sevenmsb
\scriptscriptfont\msbfam=\fivemsb

\newtheorem{theorem}{Theorem}

\newtheorem{corollary}[theorem]{Corollary}

\newtheorem{lemma}[theorem]{Lemma}

\numberwithin{equation}{section} \numberwithin{theorem}{section}

\renewcommand{\topmargin}{0cm}
\renewcommand{\oddsidemargin}{5mm}
\renewcommand{\evensidemargin}{5mm}
\renewcommand{\textwidth}{150mm}
\renewcommand{\textheight}{230mm}

\def\R{\mathbb R}

\def\S{\mathbb S}

\def\na{\nabla}

\def\f#1#2{\frac{#1}{#2}}

\def\a{\alpha}
\def\be{\beta}

\def\r{\Re_{I\!V}}

\def\p#1{\partial #1}

\def\de{\delta}
\def\De{\Delta}
\def\e{\eta}
\def\ep{\epsilon}

\def\g{\gamma}
\def\k{\kappa}
\def\la{\lambda}
\def\La{\Lambda}
\def\lan{\langle}
\def\ran{\rangle}

\def\Om{\Omega}

\def\si{\sigma}
\def\Si{\Sigma}

\def\r{\rho}

\def\div{\mathrm{div}}

\begin{document}

\title[Subsequent singularities in manifolds]
{Subsequent singularities of mean convex mean curvature flows in smooth manifolds}
\author{Qi Ding}
\address{Shanghai Center for Mathematical Sciences, Fudan University,
Shanghai 200433, China}
\email{dingqi@fudan.edu.cn}\email{dingqi09@fudan.edu.cn}

\thanks{The author is supported partially by 15ZR1402200. He would like to thank Haslhofer for discussion on non-compactness of translating solitons in the previous edition. He also would like to thank the referees for various suggestions to make this paper more readable.}

\begin{abstract}
For any $n$-dimensional smooth manifold $\Si$, we show that all the singularities of the mean curvature flow with any initial mean convex hypersurface in $\Si$ are cylindrical (of convex type) if the flow converges to a smooth hypersurface $M_\infty$ (maybe empty) at infinity.
Previously this was shown (i) for $n\le 7$, and (ii) for arbitrary $n$ up to the first singular time without the smooth condition on $M_\infty$.
\end{abstract}
\maketitle

\section{Introduction}

Singularities of mean curvature flow are unavoidable if the flow starts from a closed embedded hypersurface in Euclidean space. When the initial hypersurface is mean convex in Euclidean space, the mean curvature flow (level set flow) preserves mean convexity. So we sometimes call it mean convex mean curvature flow.

Huisken-Sinestrari obtained the convexity estimate for mean convex mean curvature flow \cite{HS1,HSi1,HSi2} and the cylindrical estimate for mean curvature flow of two-convex hypersurface \cite{HSi2}, respectively. In particular, any smooth rescaling of the singularity in the first singular time is convex by \cite{HS1,HSi1}.
B. White in \cite{W00,W03} showed that any singularity of mean convex mean curvature flow which occurs in the first singular time, must be of convex type. Here, a singular point $x$ of the flow $M_t$ has \emph{convex type} if
\begin{enumerate}
\item [(1)] any tangent flow at $x$ is cylindrical, namely, a multiplicity one shrinking round cylinder $\R^k\times\S^{n-k}$ for some $k<n$.
\item [(2)] for each sequence $x_i\in M_{t(i)}$ of regular points that converge to $x$,
$$\liminf_{i\rightarrow\infty}\f{\k_1(M_{t(i)},x_i)}{H(M_{t(i)},x_i)}\ge0,$$
\end{enumerate}
where $\k_1,\k_2\cdots,\k_n$ are the principal curvatures with $\k_1\le\k_2\cdots\le\k_n$, and $H=\sum_i\k_i>0$. Furthermore, White \cite{W03,W13} showed that all the singularities of mean convex mean curvature flow in Euclidean space are of convex type. And see \cite{And,HaK,SW} for more results in this direction.
On the other hand, Colding-Minicozzi \cite{CM1} showed that the only singularities of generic mean curvature flow in $\R^3$ are spherical or cylindrical. In \cite{CIM} Colding-Ilmanen-Minicozzi obtained a rigidity theorem for round cylinders in a very strong sense.

In the aspect of structure of the singular set of mean curvature flow, White \cite{W00} showed that parabolic Hausdorff dimension of the space-time singular set is $n-1$ at most for mean convex mean curvature flow in $\R^{n+1}$.
When a mean curvature flow starts from a closed embedded hypersurface in $\R^{n+1}$ with only generic singularities, Colding-Minicozzi \cite{CM3} showed that their space-time singular set is contained in finitely many compact embedded $(n-1)$-dimensional Lipschitz submanifolds plus a set of dimension $n-2$ at most.

When the initial hypersurface is mean convex in an $n$-dimensional smooth manifold $\Si$, mean convexity is preserved by mean curvature flow $(\mathcal{M},\mathcal{K})$ in $\Si$ in view of \cite{W00}. Let $(\mathcal{M}',\mathcal{K}')$ be any limit flow if $n\le7$ or a special
limit flow if $n>7$, where $\mathcal{K}':\, t\in\R\mapsto K_t'$ (see \cite{W03} for the definition). Then $K_t'$ is convex for every $t$ showed by White \cite{W03}. Furthermore, if $(\mathcal{M}',\mathcal{K}')$ is backwardly self-similar, then it is either (i) a static multiplicity 1 plane or (ii) a shrinking sphere or cylinder \cite{W03}.
In this paper, we will show that $K_t'$ is convex for every $t$ if $(\mathcal{M}',\mathcal{K}')$ is any limit flow for $n>7$ and the flow $\mathcal{M}$ converges to a smooth hypersurface (maybe empty) at infinity.
\begin{theorem}
Let $\mathcal{M}:\, t\in[0,\infty)\mapsto M_t$
be a mean curvature flow starting from a mean convex, smooth hypersurface in a complete smooth manifold. If $\lim_{t\rightarrow\infty}\left(\cup_{s>t}M_s\right)$ (maybe empty) is a smooth hypersurface,
then all the subsequent singularities of $\mathcal{M}$ must have convex type.
\end{theorem}
Our proof heavily depends on Ilmanen's elliptic regularization and White's work on motion by mean curvature, where we give a delicate analysis for the second fundamental form of the corresponding translating soliton related to the considered mean curvature flow in a manifold.
If either $\Si$ has nonnegative Ricci curvature or $\Si$ is simply connected with nonpositive sectional curvature, we can remove the smooth condition on the hypersurface $\lim_{t\rightarrow\infty}\left(\cup_{s>t}M_s\right)$, and get the same conclusion (see Corollary \ref{ThmCor}).
This can be thought of as a generalization of Theorem 3 of White \cite{W13}. After this paper, Haslhofer-Hershkovits \cite{HaH} got structure theorem of singularities of mean convex mean curvature flows in Riemannian manifolds by another method independently, where even they do not need the smooth condition of the flows at infinity.

\section{Translating solitons for mean curvature flow}

Let $(\Si,\si)$ be an $n$-dimensional smooth complete manifold with Riemannian metric $\si=\sum_{i,j=1}^n\si_{ij}dx_idx_j$ in a local
coordinate. Let $N$ denote the product space $\Si\times\R$ with the product metric
$$\si+dt^2=\sum_{i,j}\si_{ij}dx_idx_j+dt^2.$$
Let $\lan\cdot,\cdot\ran$ and $\overline{\na}$ denote the inner product and the Levi-Civita connection of $N$ with respect to its metric, respectively. Set $(\si^{ij})$ be the inverse matrix of $(\si_{ij})$. Let $\p_{x_i}$ and $E_{n+1}$ be the dual frame of $dx_i$ and $dt$, respectively. Denote
$Df=\sum_{i,j}\si^{ij}f_i\p_{x_j}$ and $|Df|^2=\sum_{i,j}\si^{ij}f_if_j$ for any $C^1$-function $f$ on $\Si$. Let $\div_{\Si}$ be the divergence
of $\Si$.
Let $R$ and $Ric$ denote the curvature tensor and Ricci curvature of $\Si$, respectively. Let $\overline{R}$ and $\overline{Ric}$ be the curvature tensor and the Ricci curvature of $N=\Si\times\R$, respectively.

Let $S$ be an $n$-dimensional smooth graph in $\Si\times\R$ with the graphic function $u$ and the induced metric $g$.
In a local coordinate, $g=g_{ij}dx_idx_j=\left(\si_{ij}+u_iu_j\right)dx_idx_j$, and then $g^{ij}=\si^{ij}-\f{u^iu^j}{1+|Du|^2}$, where $u^i=\si^{jk}u_k$.
Let $\De$, $\na$ be the Laplacian and Levi-Civita connection
of $(S,g)$, respectively. Let $\nu$ denote the unit normal vector field of $M$ in $N$ defined by
\begin{equation}\label{No}
\nu=\f1{\sqrt{1+|Du|^2}}(-Du+E_{n+1}).
\end{equation}

Now we assume that $S$ is a translating soliton satisfying the following equation
\begin{equation}\aligned\label{HlaEnu}
H+\la\lan E_{n+1},\nu\ran=0
\endaligned
\end{equation}
for some constant $\la>0$. The equation \eqref{HlaEnu} is equivalent to
\begin{equation}\aligned\label{divla}
\mathrm{div}_\Si\left(\f{Du}{\sqrt{1+|Du|^2}}\right)+\f{\la}{\sqrt{1+|Du|^2}}=0.
\endaligned
\end{equation}
In a local coordinate, the equality \eqref{divla} can be rewritten as
\begin{equation}\aligned\label{gupijla}
\sum_{i,j=1}^n\left(\si^{ij}-\f{u^iu^j}{1+|Du|^2}\right)u_{i,j}+\la=0,
\endaligned
\end{equation}
where $u_{i,j}$ is the covariant derivative on $\Si$ with respect to $\p_{x_i},\p_{x_j}$. Analog to Theorem 4.3 in \cite{X1}, $S$ is an area-minimizing hypersurface with the weight $e^{-\la x_{n+1}}$ in $\Si\times\R$. By $\na u=Du-\lan Du,\nu\ran\nu$ and \eqref{No}, we get
\begin{equation}\aligned\label{Enauuuu}
\left\lan E_{n+1},\na u\right\ran=-\lan Du,\nu\ran\lan E_{n+1},\nu\ran=\f{|Du|^2}{1+|Du|^2}=|\na u|^2.
\endaligned
\end{equation}

Choose a local orthonormal frame field $\{e_i\}_{i=1}^n$ in $S$, which is a normal basis at the considered point. Set the coefficients of the second fundamental form $h _{ij}=\lan\overline{\na}_{e_i}e_j,\nu\ran$
and the squared norm of the second fundamental form $|A|^2=\sum_{i,j}h_{ij}h_{ij}$. Then the mean curvature $H=\sum_ih_{ii}$.
Denote $\na_i=\na_{e_i}$ and $\overline{R}_{\nu jik}=\lan\overline{R}_{\nu j}e_i,e_k\ran=\lan-\overline{\na}_\nu\overline{\na}_{e_j}e_i+\overline{\na}_{e_j}\overline{\na}_\nu e_i+\overline{\na}_{[\nu,e_j]}e_i,e_k\ran$.
From \eqref{HlaEnu} and Codazzi equation $h_{jk,i}-h_{ji,k}=-\overline{R}_{\nu jki}$, we have
\begin{equation}\aligned\label{nainajH}
\na_i\na_jH=&-\la\na_i\lan E_{n+1},\na_{e_j}\nu\ran=\la\na_i\left(\lan E_{n+1},e_k\ran h_{jk}\right)\\
=&\la\lan E_{n+1},\nu\ran h_{ik}h_{jk}+\la\lan E_{n+1},e_k\ran h_{jk,i}\\
=&\la\lan E_{n+1},\nu\ran h_{ik}h_{jk}+\la\lan E_{n+1},e_k\ran h_{ji,k}-\la\lan E_{n+1},e_k\ran\overline{R}_{\nu jki}\\
=&\la\lan E_{n+1},\nu\ran h_{ik}h_{jk}+\la\lan E_{n+1},\na h_{ij}\ran-\la\lan E_{n+1},e_k\ran\overline{R}_{\nu jki}.
\endaligned
\end{equation}
By Simons' identity (see \cite{Z} for instance), we have
\begin{equation}\aligned\label{Simon'sIdentity}
\De h_{ij}=&\na_i\na_jH+Hh_{ik}h_{jk}-|A|^2h_{ij}+H\overline{R}_{\nu i\nu j}-h_{ij}\overline{Ric}(\nu,\nu)\\
&+\overline{R}_{kikp}h_{jp}+\overline{R}_{kjkp}h_{ip}+\overline{R}_{kijp}h_{kp}+\overline{R}_{pjik}h_{kp}+\overline{\na}_k\overline{R}_{\nu jik}+\overline{\na}_i\overline{R}_{\nu kjk}.
\endaligned
\end{equation}
From \eqref{HlaEnu}, substituting \eqref{nainajH} to \eqref{Simon'sIdentity} we get
\begin{equation}\aligned
\De h_{ij}=&\la\lan E_{n+1},\na h_{ij}\ran-|A|^2h_{ij}-\la\lan E_{n+1},e_k\ran\overline{R}_{\nu jki}+H\overline{R}_{\nu i\nu j}-h_{ij}\overline{Ric}(\nu,\nu)\\
&+\overline{R}_{kikp}h_{jp}+\overline{R}_{kjkp}h_{ip}+\overline{R}_{kijp}h_{kp}+\overline{R}_{pjik}h_{kp}+\overline{\na}_k\overline{R}_{\nu jik}+\overline{\na}_i\overline{R}_{\nu kjk}.
\endaligned
\end{equation}
Since $N$ is a product manifold with the product metric, then $\lan\overline{R}_{\nu j}e_i,E_{n+1}\ran=0$ by Appendix A of \cite{Li}.
Hence
$$-\la\lan E_{n+1},e_k\ran\overline{R}_{\nu jki}=\la\left\lan\overline{R}_{\nu j}e_i,E_{n+1}-\lan E_{n+1},\nu\ran\nu\right\ran=-\la\lan E_{n+1},\nu\ran\left\lan\overline{R}_{\nu j}e_i,\nu\right\ran=H\overline{R}_{\nu ji\nu}.$$
Then we obtain
\begin{equation}\aligned\label{Dehijla}
\De h_{ij}=&\la\lan E_{n+1},\na h_{ij}\ran-|A|^2h_{ij}-h_{ij}\overline{Ric}(\nu,\nu)\\
&+\overline{R}_{kikp}h_{jp}+\overline{R}_{kjkp}h_{ip}+2\overline{R}_{kijp}h_{kp}+\overline{\na}_k\overline{R}_{\nu jik}+\overline{\na}_i\overline{R}_{\nu kjk}.
\endaligned
\end{equation}
and
\begin{equation}\aligned\label{DeHla}
\De H=&\la\lan E_{n+1},\na H\ran-\left(|A|^2+\overline{Ric}(\nu,\nu)\right)H.
\endaligned
\end{equation}

Let $\e$ be a smooth function on $\R$ satisfying that $\e(t)=t$ for $t\in(-\infty,1]$, $1\le\e\le 2$ on $(1,3)$, $\e\equiv2$ on $[3,\infty)$, $0\le\e'\le1$ and $\e''\le0$ on $\R$. Set $\e_\a(t)=\a\e\left(\f{t}{\a}\right)$ on $\R$ for any $\a>0$. Assume that $\Om$ is a domain in $\Si$ with positive mean curvature boundary $\p\Om$. Then there is an area-minimizing hypersurface $S_{\la,\a}$ with the weight $e^{-\la \e_\a(x_{n+1})}$ in $\Om\times\R$ with $\p S_{\la,\a}=\p\Om\times\{0\}$ for $\la>0$. Note that $\e_\a$ has a finite upper bound, so it is not hard to see that $S_{\la,\a}$ is compact. In particular, $S_{\la,\a}$ satisfies the following equation
\begin{equation}\aligned\label{Hlaa}
H+\la\e'_\a\lan E_{n+1},\nu\ran=0,
\endaligned
\end{equation}
where $\nu$ is the unit normal vector of $S_{\la,\a}$.
If $S_{\la,\a}$ is a graph over some open set with the graphic function $w$, the equation \eqref{Hlaa} can be rewritten as
\begin{equation}\aligned\label{gupijlae'}
\sum_{i,j=1}^n\left(\si^{ij}-\f{w^iw^j}{1+|Dw|^2}\right)w_{i,j}+\la\e'_\a(w)=0
\endaligned
\end{equation}
in a local coordinate.

Now let's show that $S_{\la,\a}$ is a graph over $\Om$ by White's argument \cite{W13}. Let $S_{\la,\a}(t)=\{X-tE_{n+1}|\ X\in S_{\la,\a}\}$ for any $t\in\R$. Let $t_m$ be the smallest nonnegative number such that $S_{\la,\a}(t_m)$ intersects $S_{\la,\a}$ in the interior. If $t_m=0$, then clearly $S_{\la,\a}$ is a graph. If $t_m>0$, the two hypersurfaces $S_{\la,\a}$ and $S_{\la,\a}(t_m)$ touch at an interior point $X=(x,t)\in\Om\times\R$. Hence the tangent cones to $S_{\la,\a}$ and to $S_{\la,\a}(t_m)$ at $X$, respectively, both lie in halfspaces and are therefore a same hyperplane. So $X$ is a regular point of $S_{\la,\a}$ and $S_{\la,\a}(t_m)$, respectively. $S_{\la,\a}(t_m)$ satisfies the equation
$$H+\la\e'_\a(\cdot+t_m)\lan E_{n+1},\nu\ran=0.$$
Namely, $S_{\la,\a}(t_m)$ can be written as a graph over some neighborhood of $x$ with a graphic function $\tilde{w}$ satisfying
\begin{equation}\aligned
\sum_{i,j=1}^n\left(\si^{ij}-\f{\tilde{w}^i\tilde{w}^j}{1+|D\tilde{w}|^2}\right)\tilde{w}_{i,j}+\la\e'_\a(\tilde{w}+t_m)=0.
\endaligned
\end{equation}
Let $\varphi=w-\tilde{w}$, then $\varphi=0$ at $x$ and $\varphi\ge0$ in some neighborhood of $x$ as $S_{\la,\a}(t_m)$ is below $S_{\la,\a}$. Moreover,
\begin{equation}\aligned\label{Compula'}
&\sum_{i,j}\left(\si^{ij}-\f{\p^iw\p^jw}{1+|Dw|^2}\right)\varphi_{i,j}
=-\la\e'_\a(w)-\sum_{i,j}\left(\si^{ij}-\f{\p^iw\p^jw}{1+|Dw|^2}\right)\tilde{w}_{i,j}\\
=&-\la\e'_\a(w)+\la\e'_\a(\tilde{w}+t_m)+\sum_{i,j}\left(\f{\p^iw\p^jw}{1+|Dw|^2}-\f{\p^i\tilde{w}\p^j\tilde{w}}{1+|D\tilde{w}|^2}\right)\tilde{w}_{i,j}\\
=&-\la\e'_\a(w)+\la\e'_\a(\tilde{w}+t_m)+\sum_{i,j}\left(\f{\p^iw\p^jw-\p^i\tilde{w}\p^j\tilde{w}}{1+|D\tilde{w}|^2}
+\f{\p^iw\p^jw}{1+|Dw|^2}-\f{\p^iw\p^jw}{1+|D\tilde{w}|^2}\right)\tilde{w}_{i,j}\\
=&-\la\e'_\a(w)+\la\e'_\a(\tilde{w}+t_m)+\sum_{i,j}\f{\p^i\varphi\p^jw+\p^i\tilde{w}\p^j\varphi}{1+|D\tilde{w}|^2}\tilde{w}_{i,j}\\
&\quad-\f{\p^iw\p^jw\ \tilde{w}_{i,j}}{\left(1+|Dw|^2\right)\left(1+|D\tilde{w}|^2\right)}\lan D(w+\tilde{w}),D\varphi\ran.
\endaligned
\end{equation}
Combining $\e''_\a\le0$ and $t_m>0$, we conclude that $S_{\la,\a}$ and $S_{\la,\a}(t_m)$ lie in parallel vertical planes by the strong maximum principle at $x$. However, it is impossible. So $S_{\la,\a}$ is a smooth graph over $\Om$, which minimizes the weighted area with the weight $e^{-\la\e_\a(x_{n+1})}$ and with boundary $\p\Om\times\{0\}$ in $\Om\times\R$. Moreover, if $u_{\la,\a}$ is the graphic function of $S_{\la,\a}$, then $u_{\la,\a}\ge0$ by maximum principle.

Set $d(x)=d(x,\p\Om)$ for all $x\in\Om$, and $\Om_t=\{x\in\Om|\ d(x)>t\}$ for $t\ge0$. There is a constant $0<\ep<1$ such that $\p\Om_t$ is smooth with mean curvature $\mathscr{H}(x,t)\ge\ep$ for any $x\in\p\Om_t$ and $0\le t<\ep$, and $d$ is smooth on $\Om\setminus\Om_\ep$.
\begin{lemma}
Let $u_{\la,\a}$ be a smooth solution to \eqref{gupijlae'} on $\overline{\Om}$ with $u_{\la,\a}=0$ on $\p\Om$ for any $\la,\a>0$, then we have the following boundary gradient estimate:
$$|Du_{\la,\a}|\le\ep^{-1}(\la+1)\qquad \mathrm{on}\ \ \p\Om.$$
\end{lemma}
\begin{proof}
Let $\{e_i\}_{i=1}^{n}$ be an orthonormal vector field tangent to $\p\Om_t$ at a considered point $x\in\p\Om_t$, and denote $e_n$ be the unit normal vector field to $\p\Om_t$ so that $e_n$ points into $\Om_t$. Since $d$ is a constant on $\p\Om_t$, then at $x$ we get
\begin{equation}\aligned
\sum_{i=1}^{n-1}\left(D_{e_i}D_{e_i}-(D_{e_i}e_i)^T\right)d=0,
\endaligned
\end{equation}
and $\left(D_{e_n}D_{e_n}-D_{e_n}e_n\right)d=0$, where $(\cdots)^T$ denotes the projection into the tangent bundle of $\p\Om_t$.
Hence at $x$ one has
\begin{equation}\aligned\label{DeSid}
\De_\Si d=&\sum_{i=1}^n\left(D_{e_i}D_{e_i}-D_{e_i}e_i\right)d=-\sum_{i=1}^{n-1}\left\lan D_{e_i}e_i,e_n\right\ran D_{e_n}d+\left(D_{e_n}D_{e_n}-D_{e_n}e_n\right)d\\
=&-\sum_{i=1}^{n-1}\left\lan D_{e_i}e_i,e_n\right\ran=-\mathscr{H}(x,t).
\endaligned
\end{equation}

Set $\Phi=\phi(d)=-(\la+1)\log\left(1-\ep^{-1}d\right)$. Then $\phi'=(\la+1)(\ep-d)^{-1}$ and $\phi''=(\la+1)(\ep-d)^{-2}$. Together with \eqref{divla}, \eqref{DeSid} and $\mathscr{H}\ge\ep$, on $\Om\setminus\Om_{\ep}$ we conclude that
\begin{equation}\aligned
&\mathrm{div}_\Si\left(\f{D\Phi}{\sqrt{1+|D\Phi|^2}}\right)=\mathrm{div}_\Si\left(\f{\phi'Dd}{\sqrt{1+|\phi'|^2}}\right)
=\f{\phi'}{\sqrt{1+|\phi'|^2}}\De_\Si d+\f{\phi''}{\left(1+|\phi'|^2\right)^{\f32}}\\
=&\f{-\mathscr{H}}{\sqrt{1+(\la+1)^{-2}(\ep-d)^2}}+\f{(\la+1)^{-2}(\ep-d)}{\left(1+(\la+1)^{-2}(\ep-d)^2\right)^{\f32}}\\
\le&\f{-\ep}{\sqrt{1+(\la+1)^{-2}(\ep-d)^2}}+\f{\ep(\la+1)^{-2}}{\left(1+(\la+1)^{-2}(\ep-d)^2\right)^{\f32}}
\le\f{-\la(\la+1)^{-1}\ep}{\sqrt{1+(\la+1)^{-2}(\ep-d)^2}}\\
\le&\f{-\la(\la+1)^{-1}(\ep-d)}{\sqrt{1+(\la+1)^{-2}(\ep-d)^2}}=\f{-\la}{\sqrt{1+|D\Phi|^2}}.
\endaligned
\end{equation}
Assume that $u_{\la,\a}$ is a smooth solution to \eqref{gupijlae'} on $\overline{\Om}$ with $u_{\la,\a}=0$ on $\p\Om$ for any $\la,\a>0$.
Analog to \eqref{Compula'}, $0\le u_{\la,\a}\le \Phi$ on $\Om\setminus\Om_{\ep}$ by comparison principle. Then
\begin{equation}\aligned\label{uwcomp}
|Du_{\la,\a}|\le|D\Phi|=\ep^{-1}(\la+1)\qquad \mathrm{on}\ \ \p\Om.
\endaligned
\end{equation}
\end{proof}

For $\a>0$,
\begin{equation}\aligned
\f{\p}{\p\a}\e_\a(t)=\e\left(\f t\a\right)-\f t\a\e'\left(\f t\a\right)=-\int_0^{\f t\a}\tau\e''(\tau)d\tau\ge0.
\endaligned
\end{equation}
Similar to \eqref{Compula'}, we conclude that $S_{\la,\be}$ is above $S_{\la,\a}$ for any $\be>\a$. Letting $\a\rightarrow\infty$, the graph $S_{\la,\a}$ converges to a generalized graph $S_\la$ over $\Om$. Namely, let $u_\la=\lim_{\a\rightarrow\infty}u_{\la,\a}\ge0$, then $S_\la=\p\{(x,t)|\ u_\la(x)<t\}$ and $u_{\la}$ is the graphic function of $S_{\la}$. Note that $\{u_\a=\infty\}$ may be not empty, i.e., $S_\la$ may be not compact. Moreover, $S_\la$ is smooth and satisfies the equation \eqref{HlaEnu} on $\{u_\a<\infty\}$. Obviously, $S_\la$ is an area-minimizing hypersurface with the weight $e^{-\la x_{n+1}}$ in $\Om\times\R$ with boundary $\p\Om\times\{0\}$.


\section{Proof of the main theorem}

Let $\Om$ be a bounded domain in an $n$-dimensional smooth manifold $\Si$ with smooth mean convex boundary. Assume that $\p\Om$ is not a minimal hypersurface in $\Si$.
From \cite{W00}, there is a mean curvature (level set) flow $\mathcal{M}:\ t\in[0,\infty)\mapsto M_t$ with $M_0=\p\Om$. By maximum principle, there is a sufficiently small constant $\ep_0>0$ such that $M_t$ has positive mean curvature everywhere for all $0<t\le\ep_0$. Without loss of generality, we assume that $\p\Om$ has positive mean curvature everywhere.

Denote $F_t(\Om)$ be a domain in $\Om$ with $\p F_t(\Om)=M_t$ for $t\in[0,\infty)$.
By \cite{W00}, $F_t(\Om)$ is mean convex for each $t\in[0,\infty)$, and $M_t\cap M_{t+\tau}=\emptyset$,
$F_{t+\tau}(\Om)\subset \mathrm{interior}(F_t(\Om))$ for all $0\le t<t+\tau<\infty$.
Let $v:\bigcup_{t\ge0}M_t\rightarrow\R$ be the function such that $v(x)=t$ for each $x\in M_t$. Then $v$ satisfies
\begin{equation}\aligned\label{levelsv}
\mathrm{div}_\Si\left(\f{Dv}{|Dv|}\right)+\f{1}{|Dv|}=0
\endaligned
\end{equation}
in the viscosity sense.
Set $\Om_\infty=\bigcap_{t>0}F_t(\Om)$ and $M_\infty=\p\Om_\infty$. By \cite{W00}, $M_\infty$(maybe empty) has finitely many connected components, and the boundary of each component is a stable minimal variety whose singular set has Hausdorff dimension $\le n-8$.
Let the parabolic Hausdorff dimension of a set $E\subset\Si\times\R$ be the Hausdorff dimension of $E$ with respect to parabolic distance
$$\mathrm{dist}_{P}((x,t),(y,\tau))=\max\{d(x,y),|t-\tau|^{1/2}\},$$
where $d(\cdot,\cdot)$ is the distance function on $\Si$.
Let $\widetilde{\mathcal{S}}$ be the spacetime singular set of $\mathcal{M}$ defined in \cite{W00}. Then the parabolic Hausdorff dimension of $\widetilde{\mathcal{S}}$ is at most $n-2$ by \cite{W00}. Denote $\mathcal{S}=\{x\in\Om|\ \mathrm{there\ exists\ a\ } t>0\ \mathrm{such\ that}\ (x,t)\in \widetilde{\mathcal{S}}\}$.

Now we assume that $M_\infty$ is smooth. Then the mean curvature flow $M_t$ converges smoothly $M_\infty$ as $t\rightarrow\infty$. So there is an open set $K$ with $\overline{K}\subset\Om\setminus\Om_\infty$ such that $\overline{\mathcal{S}}\subset K$ and $v$ is smooth on $\overline{K}\setminus\mathcal{S}$.

From the previous argument, there is a translating soliton $S_\la$ whose graphic function $u_\la$ is a generalized function on $\Om$ satisfying \eqref{divla} on $\{u_\la<\infty\}$ and $u_\la=0$ on $\p\Om$ for any $\la>0$.
Then
$$t\in\R\rightarrow(S_\la)_t\triangleq \mathrm{graph}(u_\la-\la t)$$
is a family of smooth hypersurfaces in $\Om\times\R$ moving by mean curvature. Analog to the proof of Theorem 3 in \cite{W13}, set $U_\la:\,\Om\times\R\rightarrow\R$ by
$$U_\la(x,y)=\la^{-1}\left(u_\la(x)-y\right),$$
and $U:\,\Om\times\R\rightarrow\R$ by
$$U(x,y)=v(x).$$
As $\la\rightarrow\infty$, the mean curvature flows $t\in[0,\infty)\rightarrow(S_\la)_t$ converge as brakke flows to the flow $t\rightarrow M_t\times\R$ by elliptic regularization \cite{I94} and uniqueness of viscosity solution $v$. Since $U_\la^{-1}(t)\cap\{y=\tau\}=(S_\la)_t\cap\{y=\tau\}$ and $v^{-1}(t)\cap\{y=\tau\}=M_t$ for all $t,\tau\ge0$, then $U_\la$ converges as $\la\rightarrow\infty$ uniformly to $U$ on $\overline{K}$. Namely, $\la^{-1}u_\la$ converges uniformly to $v$ on $\overline{K}$. By the local regularity theorem in \cite{W05} (or by Brakke's regularity theorem in \cite{Br}), $\la^{-1}u_\la$ converges as $\la\rightarrow\infty$ to $v$ smoothly on $\overline{K}\setminus\mathcal{S}$.

Let $H_\la$ be the mean curvature of $S_\la$. By
$$H_\la=-\f{\la}{\sqrt{1+|Du_\la|^2}}=-\f{1}{\sqrt{\la^{-2}+\la^{-2}|Du_\la|^2}},$$
and $\la^{-1}u_\la$ converges uniformly to $v$ on $\overline{K}$,
there is a small constant $0<\de<1$ independent of $\la\ge1$
such that
\begin{equation}\aligned\label{uplowerboundHs}
H_\la\le-\de\qquad \mathrm{on}\ \overline{K}\ \ \mathrm{for\ every}\ \la\ge1.
\endaligned
\end{equation}
Denote $|A_\la|^2$ be the square norm of the second fundamental form of $S_\la$. Choose a local orthonormal frame field $\{e_i\}_{i=1}^n$ in $S_\la$, which is a normal basis at the considered point.
Combining \eqref{Dehijla} and \eqref{DeHla}, for any constant $\g$ we obtain
\begin{equation}\aligned
\De \left(h_{ij}+\g H_\la\right)=&\left\lan\la E_{n+1},\na \left(h_{ij}+\g H_\la\right)\right\ran-\left(|A_\la|^2+\overline{Ric}(\nu,\nu)\right)\left(h_{ij}+\g H_\la\right)\\
&+\overline{R}_{kikp}h_{jp}+\overline{R}_{kjkp}h_{ip}+2\overline{R}_{kijp}h_{kp}+\overline{\na}_k\overline{R}_{\nu jik}+\overline{\na}_i\overline{R}_{\nu kjk}
\endaligned
\end{equation}
on $K$. Obviously, $|A_\la|^2\ge\f1n|H_\la|^2\ge\f1n\de^2$ on $\overline{K}$ by \eqref{uplowerboundHs},
then there is a positive constant $C_0$ depending only on $n$, $\de$, $|R|$ and $|DR|$ on $\Om$ such that
\begin{equation}\aligned\label{hijgHla}
\De \left(h_{ij}+\g H_\la\right)\ge&\left\lan\la E_{n+1},\na \left(h_{ij}+\g H_\la\right)\right\ran-\left(|A_\la|^2+\overline{Ric}(\nu,\nu)\right)\left(h_{ij}+\g H_\la\right)
-C_0|A_\la|
\endaligned
\end{equation}
on $K$. Here $|R|^2=\sum_{i,j,k,l}|R_{ijkl}|^2$ and $|DR|^2=\sum_{i,j,k,l,m}|DR_{ijkl,m}|^2$.

\begin{lemma}\label{BonndAHla}
There is a positive constant $\g^*_\la\ge1$ depending only on $n$, $\de$, $|R|$, $|DR|$ on $\Om$ and $\inf_{\p K}\left(|A_\la|H_\la^{-1}\right)$ such that
\begin{equation}\aligned
-\f1{\g^*_\la}H_\la\le|A_\la|\le-\g^*_\la H_\la\qquad \mathrm{on}\ \ \overline{K}.
\endaligned
\end{equation}
\end{lemma}
\begin{proof}
Let $\k_1\ge\k_2\cdots\ge\k_n$ be the principal curvatures of $S_\la$. Note that $\k_1=\sup_{|\xi|=1}A_\la(\xi,\xi)$, then $\k_1$ is a continuous function on $S_\la$. Further, for any $\g,\tilde{\g}\in\R$,
$$\sup_{K}\left(\k_1+\g H_\la\right)\le\sup_{K}\left(\k_1+\tilde{\g} H_\la\right)+\sup_{K}\left((\g-\tilde{\g})H_\la\right),$$
which implies that $\sup_{K}\left(\k_1+\g H_\la\right)$ is also a continuous function on $\g\in\R$. There is a constant $\g_0$ depending on $\inf_{\p K}\left(|A_\la|H_\la^{-1}\right)$ such that
$$\sup_{\p K}\left(\k_1+\g_0H_\la\right)=0.$$
If $\g_0<0$, we reset $\g_0=0$.
Then we choose a constant $\g_1$ such that
$$\sup_{K}\left(\k_1+\g_1 H_\la\right)=-1.$$
We assume $\g_1>\g_0+\f1{\de}$, or else we complete the proof. By $H_\la\le-\de$, on $\p K$ we have
\begin{equation}\aligned
\k_1+\g_1 H_\la=\k_1+\g_0 H_\la+(\g_1-\g_0)H_\la\le(\g_1-\g_0)H_\la\le-(\g_1-\g_0)\de<-1.
\endaligned
\end{equation}
Hence $\k_1+\g_1 H$ attains its maximum at some point $x_0$ in the interior of $K$. Choose a local orthonormal frame $\{e_i\}$ near $x_0$ in $\Si$ which is normal at $x_0$, and denote $h_{ij}=h(e_i,e_j)$ as mentioned before. Let $\xi=\sum_i\xi_ie_i\big|_{p_0}$ be a unit eigenvector of the second fundamental form corresponding to the eigenvalue $\k_1(x_0)$ at the point $x_0$, namely, $h(\xi,\xi)=\k_1(x_0)$. Then the smooth function $\hat{\k}_1\triangleq\sum_{i,j}h_{ij}\big|_x\xi_i\xi_j$ attains the maximum $\k_1(x_0)$ at $x_0$ in a neighborhood of $x_0$. From \eqref{hijgHla}, we obtain
\begin{equation}\aligned\label{k1gHla}
\De \left(\hat{\k}_1+\g H_\la\right)\ge&\left\lan\la E_{n+1},\na \left(\hat{\k}_1+\g H_\la\right)\right\ran-\left(|A_\la|^2+\overline{Ric}(\nu,\nu)\right)\left(\hat{\k}_1+\g H_\la\right)
-C_0|A_\la|.
\endaligned
\end{equation}
By maximum principle for \eqref{k1gHla}, at $x_0$ we have
\begin{equation}\aligned\label{|A|2Ric}
0\ge-\left(|A_\la|^2+\overline{Ric}(\nu,\nu)\right)\left(\hat{\k}_1+\g_1 H_\la\right)-C_0|A_\la|=|A_\la|^2+\overline{Ric}(\nu,\nu)-C_0|A_\la|.
\endaligned
\end{equation}
Let $c_0=\min\left\{0,\inf_{|\xi|=1,x\in\Om}Ric\big|_{x}(\xi,\xi)\right\}$. Then \eqref{|A|2Ric} implies that at $x_0$
\begin{equation}\aligned\label{est|A|}
|A_\la|\le\f{C_0}2+\sqrt{\f{C_0^2}4-c_0}\le C_0+\sqrt{-c_0}.
\endaligned
\end{equation}
On the other hand, by \eqref{uplowerboundHs} at $x_0$ one has
\begin{equation}\aligned\label{-1kgH}
-1=\k_1+\g_1H_\la\le|A_\la|+\g_1H_\la\le|A_\la|-\g_1\de.
\endaligned
\end{equation}
Combining \eqref{est|A|}\eqref{-1kgH} and the assumption $\g_1>\g_0+\f1{\de}$, we obtain
\begin{equation}\aligned
\g_1\le\g_0+\f1\de\left(C_0+\sqrt{-c_0}+1\right).
\endaligned
\end{equation}
According to the definition of $\g_1$ and $\k_i$, $\k_1+\g_1 H_\la\le-1<0$ on $\overline{K}$, which implies that
\begin{equation}\aligned
\k_n=H_\la-\sum_{i=1}^{n-1}\k_i\ge H_\la-(n-1)\k_1\ge\left(1+(n-1)\g_1\right)H_\la.
\endaligned
\end{equation}
Hence we complete the proof.
\end{proof}

Due to
\begin{equation}\aligned
&\left\lan\overline{\na}_{\p_{x_i}+\p_iu_\la E_{n+1}}(\p_{x_j}+\p_ju_\la E_{n+1}),\f{-Du_\la+E_{n+1}}{\sqrt{1+|Du_\la|^2}}\right\ran\\
=&\left\lan D_{\p_{x_i}}\p_{x_j},\f{-Du_\la+E_{n+1}}{\sqrt{1+|Du_\la|^2}}\right\ran+\p_{x_i}\p_{x_j}u_\la\left\lan E_{n+1},\f{-Du_\la+E_{n+1}}{\sqrt{1+|Du_\la|^2}}\right\ran\\
=&\left(\p_{x_i}\p_{x_j}u_\la-\lan D_{\p_{x_i}}\p_{x_j},Du_\la\ran\right)\f1{\sqrt{1+|Du_\la|^2}}=\f{(u_\la)_{i,j}}{\sqrt{1+|Du_\la|^2}},
\endaligned
\end{equation}
we have
\begin{equation}\aligned\label{Ala2}
|A_\la|^2=\sum_{i,j,k,l}\left(\si^{ij}-\f{\p^iu_\la\p^ju_\la}{1+|Du_\la|^2}\right)
\f{(u_\la)_{j,k}}{\sqrt{1+|Du_\la|^2}}\left(\si^{kl}-\f{\p^ku_\la\p^lu_\la}{1+|Du_\la|^2}\right)
\f{(u_\la)_{l,i}}{\sqrt{1+|Du_\la|^2}}.
\endaligned
\end{equation}

Now let's show the main theorem.
\begin{theorem}\label{mainThm}
Let $\mathcal{M}:\, t\in[0,\infty)\mapsto M_t$
be a mean curvature flow starting from a mean convex, smooth hypersurface in an $n$-dimensional complete smooth manifold $\Si$. If $\lim_{t\rightarrow\infty}\left(\cup_{s>t}M_s\right)$ (maybe empty) is a smooth hypersurface,
then all the singularities of $\mathcal{M}$ have convex type.
\end{theorem}
\begin{proof}
Let $v$ be a viscosity solution to \eqref{levelsv} on $\Om\setminus\Om_\infty$, then $|Dv|>0$ on $K\setminus \mathcal{S}$.
From \eqref{Ala2}, $H_\la$ converges to $\mathrm{div}_\Si\left(\f{Dv}{|Dv|}\right)$, and
\begin{equation}\aligned
|A_\la|^2&=\sum_{i,j,k,l}\left(\si^{ij}-\f{\p^iU_\la\p^jU_\la}{\la^{-2}+|DU_\la|^2}\right)
\left(\si^{kl}-\f{\p^kU_\la\p^lU_\la}{\la^{-2}+|DU_\la|^2}\right)
\f{(U_\la)_{j,k}(U_\la)_{l,i}}{\la^{-2}+|DU_\la|^2}\\
\rightarrow|A_\infty|^2&\triangleq\sum_{i,j,k,l}\left(\si^{ij}-\f{\p^iv\p^jv}{|Dv|^2}\right)
\f{v_{j,k}}{|Dv|}\left(\si^{kl}-\f{\p^kv\p^lv}{|Dv|^2}\right)
\f{v_{l,i}}{|Dv|}\qquad \mathrm{as}\ \la\rightarrow\infty
\endaligned
\end{equation}
on $K\setminus \mathcal{S}$ smoothly. Here $-\mathrm{div}_\Si\left(\f{Dv}{|Dv|}\right)$ and $|A_\infty|^2$ are the mean curvature and the square norm of the second fundamental form for the level set of $v$ in $K\setminus \mathcal{S}$, repsectively. Since $\p K\cap\overline{\mathcal{S}}=\emptyset$, we conclude that $\inf_{\p K}\left(|A_\la|H_\la^{-1}\right)$ is uniformly bounded for any $\la\ge1$, and then $\g^*_\la$ in Lemma \ref{BonndAHla} is bounded by an absolute constant $\g^*$ independent of $\la\ge1$. Namely, by Lemma \ref{BonndAHla} we have
$$-\f1{\g^*}H_\la\le|A_\la|\le-\g^* H_\la\qquad \mathrm{on}\ \ \overline{K}.$$
Hence we obtain that
\begin{equation}\aligned
-\f1{\g^*}\mathrm{div}_\Si\left(\f{Dv}{|Dv|}\right)\le|A_\infty|\le-\g^*\mathrm{div}_\Si\left(\f{Dv}{|Dv|}\right)\qquad \mathrm{on}\ K\setminus \mathcal{S}.
\endaligned
\end{equation}
According to appendix B in \cite{W13}, we complete the proof.
\end{proof}

(i) Assume that $\Si$ has nonnegative Ricci curvature in Theorem \ref{mainThm}. From \eqref{nainajH}, we have
\begin{equation}\aligned
\De\lan E_{n+1},\nu\ran=-\lan E_{n+1},\na H\ran-\big(|A|^2+\overline{Ric}(\nu,\nu)\big)\lan E_{n+1},\nu\ran.
\endaligned
\end{equation}
Let $S_{\la,\a}$ be a smooth graph with the graphic function $u_{\la,\a}$ satisfying \eqref{Hlaa}.
Combining \eqref{Enauuuu}, one has
\begin{equation}\aligned
\De\lan E_{n+1},\nu\ran=\la\e'_\a\lan E_{n+1},\na \lan E_{n+1},\nu\ran\ran-\big(|A|^2+\overline{Ric}(\nu,\nu)-\la\e''_\a|\na u_{\la,\a}|^2\big)\lan E_{n+1},\nu\ran.
\endaligned
\end{equation}
Then by maximum principle for the above equation we have
$$\sup_\Om\sqrt{1+|Du_{\la,\a}|^2}\le\sup_{\p\Om}\sqrt{1+|Du_{\la,\a}|^2}.$$
Combining the estimate \eqref{uwcomp}, $\f1{\la+1}|Du_{\la,\a}|$ is uniformly bounded on $\Om$ independent of $\la,\a>0$ as well as $\f1{\la+1}u_{\la,\a}$. Hence the graph$_{u_{\la}}$ is a bounded graph with $u_\la=\lim_{\a\rightarrow\infty}u_{\la,\a}$.
Since $\f1{\la}u_\la$ converges to $v$ as $\la\rightarrow\infty$ on any compact set $Q$ in $\Om\setminus\Om_\infty$, we get that $v$ is bounded on $Q$ by a constant independent of $Q$. Hence the mean curvature flow $\mathcal{M}$ in Theorem \ref{mainThm} must vanish in finite time.

(ii) If $\Si$ is simply connected with nonpositive sectional curvature in Theorem \ref{mainThm}, then we claim
\begin{equation}\aligned\label{1t1Dut}
\sup_{t\in(0,\infty)}\left((1+t)^{-1}\sup_{x\in\Om}u_{t}(x)\right)<\infty.
\endaligned
\end{equation}
Let's prove it by contradiction.
If \eqref{1t1Dut} fails, there are sequence $t_i>0$ and $\a_i\rightarrow\infty$ such that $(1+t_i)^{-1}\sup_{\Om}u_{t_i,\a_i}\rightarrow\infty$ as $i\rightarrow\infty$, where $u_{t_i,\a_i}$ is a smooth solution to \eqref{gupijlae'} with $\la,\a$ replaced by $t_i,\a_i$, respectively. We define $s_i\triangleq\sup_{\Om}u_{t_i,\a_i}$ and $\widehat{u}_{t_i}=s_i^{-1}u_{t_i,\a_i}$, then 
\begin{equation}\aligned\label{divuti}
\mathrm{div}_\Si\left(\f{D\widehat{u}_{t_i}}{\sqrt{s_i^{-2}+|D\widehat{u}_{t_i}|^2}}\right)+\f{t_i\e'_{\a_i}}{s_i\sqrt{s_i^{-2}+|D\widehat{u}_{t_i}|^2}}=0.
\endaligned
\end{equation}
On the other hand, there is a point $x_{t_i}\in\Om$ such that $\widehat{u}_{t_i}(x_{t_i})=1$.
Let $\r_{x_{t_i}}(x)=d(x,x_{t_i})$ for any $x\in\Om$, then $\r_{x_{t_i}}^2$ is smooth on $\Si$. By Hessian comparison theorem, we have
$$\De_\Si \r_{x_{t_i}}^2\ge 2n.$$
Set $\La=2\mathrm{diam}(\Om)>0$.
Note $(1+t_i)^{-1}s_i\rightarrow\infty$ as $i\rightarrow\infty$. Hence for sufficiently large $i>0$
\begin{equation}\aligned\label{divrxti}
&\mathrm{div}_\Si\left(\f{D\left(\f12-\La^{-2}\r_{x_{t_i}}^2\right)}{\sqrt{s_i^{-2}+\left|D\left(\f12-\La^{-2}\r_{x_{t_i}}^2\right)\right|^2}}\right)+
\f{t_i}{s_i\sqrt{s_i^{-2}+\left|D\left(\f12-\La^{-2}\r_{x_{t_i}}^2\right)\right|^2}}\\
\le&\f{-\La^{-2}\De_\Si\r_{x_{t_i}}^2}{\sqrt{s_i^{-2}+4\La^{-4}\r_{x_{t_i}}^2}}+\f{8\La^{-6}\r_{x_{t_i}}^2}{\left(s_i^{-2}+4\La^{-4}\r_{x_{t_i}}^2\right)^{\f32}}
+\f{t_i}{s_i\sqrt{s_i^{-2}+4\La^{-4}\r_{x_{t_i}}^2}}\\
\le&-\f{2n\La^{-2}}{\sqrt{s_i^{-2}+4\La^{-4}\r_{x_{t_i}}^2}}+\f{8\La^{-6}\r_{x_{t_i}}^2}{\left(s_i^{-2}+4\La^{-4}\r_{x_{t_i}}^2\right)^{\f32}}
+\f{t_i}{s_i\sqrt{s_i^{-2}+4\La^{-4}\r_{x_{t_i}}^2}}\\
\le&\f{-2(n-1)\La^{-2} s_i+t_i}{s_i\sqrt{s_i^{-2}+4\La^{-4}\r_{x_{t_i}}^2}}<0.
\endaligned
\end{equation}
Let $\mathcal{E}$ be an open set defined by $\{x\in\Om|\, \widehat{u}_{t_i}>\f12-\La^{-2}\r_{x_{t_i}}^2\}$.
Since $\widehat{u}_{t_i}(x_{t_i})=1$ and $\f12-\La^{-2}\r_{x_{t_i}}^2>0=\widehat{u}_{t_i}$ on $\p\Om$, then $\widehat{u}_{t_i}-\f12+\La^{-2}\r_{x_{t_i}}^2=0$ on $\p\mathcal{E}$.
Analog to \eqref{Compula'},
$\widehat{u}_{t_i}-\f12+\La^{-2}\r_{x_{t_i}}^2$ attains its maximum on $\mathcal{E}$ at the boundary $\p\mathcal{E}$ by the maximum principle for \eqref{divuti} and \eqref{divrxti}. So we get a contradiction as $\widehat{u}_{t_i}-\f12+\La^{-2}\r_{x_{t_i}}^2=0$ on $\p\mathcal{E}$, and the claim \eqref{1t1Dut} holds.

So $(1+t)^{-1}\sup_{\Om}|u_t|$ is uniformly bounded independent of $t>0$, which implies that $v$ is bounded and the mean curvature flow $\mathcal{M}$ in Theorem \ref{mainThm} must vanish in finite time.

Therefore, from Theorem \ref{mainThm} we can get the following corollary.
\begin{corollary}\label{ThmCor}
Let $\mathcal{M}:\, t\in[0,\infty)\mapsto M_t$
be a mean curvature flow starting from a mean convex, smooth hypersurface in an $n$-dimensional complete smooth manifold $\Si$. If either $\Si$ has nonnegative Ricci curvature or $\Si$ is simply connected with nonpositive sectional curvature,
then all the singularities of $\mathcal{M}$ have convex type.
\end{corollary}

\bibliographystyle{amsplain}

\end{document}